\documentclass[reqno, a4paper, 12pt]{amsart}

\usepackage{stmaryrd}
\usepackage{mathrsfs}
\usepackage{multirow}
\usepackage[english]{babel}
\usepackage[utf8]{inputenc}
\usepackage{amsmath, tabu}
\usepackage{amsthm}
\usepackage{amssymb}
\usepackage{tikz-cd}
\usepackage{graphicx}
\usepackage[colorinlistoftodos]{todonotes}
\usepackage{enumerate}
\usepackage{hyperref}
\usepackage{extarrows}
\usepackage{bm}
\usepackage{xcolor}
\usepackage{soul}
\usepackage{url}
\hypersetup{
	colorlinks,
	linkcolor={red!100!black},
	citecolor={blue!100!black},
	urlcolor={green!100!black}
}
\newtheorem{thm}{Theorem}[section]

\newtheorem{lemma}[thm]{Lemma}
\newtheorem{cor}[thm]{Corollary}

\theoremstyle{definition}
\newtheorem{defn}[thm]{Definition}

\newtheorem{ex}[thm]{Example}
\theoremstyle{remark}
\newtheorem{rmk}[thm]{Remark}

\newtheoremstyle{named}{}{}{\itshape}{}{\bfseries}{.}{.5em}{#1 \thmnote{#3}}
\theoremstyle{named}

\newtheoremstyle{named}{}{}{\itshape}{}{\bfseries}{.}{.5em}{#1 \thmnote{#3}}
\theoremstyle{named}

\newtheoremstyle{named}{}{}{\itshape}{}{\bfseries}{.}{.5em}{#1 \thmnote{#3}}
\theoremstyle{named}

\newcommand{\Q}{\mathbb{Q}}
\newcommand{\Z}{\mathbb{Z}}

\newcommand{\C}{\mathbb{C}}

\newcommand{\frakp}{\mathfrak{p}}

\newcommand{\frakN}{\mathfrak{N}}

\newcommand{\scrL}{\mathscr{L}}

\newcommand{\calO}{\mathcal{O}}
\newcommand{\Sel}{\text{Sel}}

\newcommand{\cyc}{\text{cyc}}

\DeclareMathOperator{\Gal}{Gal}

\usepackage{xcolor}
\usepackage[foot]{amsaddr}
\usepackage{academicons}
\newcommand{\orcid}[1]{\href{https://orcid.org/#1}{\textcolor[HTML]{A6CE39}{\aiOrcid}}}

\title{On $\lambda$-invariants of congruent modular forms in the anticyclotomic indefinite setting}
\author[Dac-Nhan-Tam Nguyen]{Dac-Nhan-Tam Nguyen}
\email{tamnguyen@math.ubc.ca}
\address{Department of Mathematics, University of British Columbia \\
	1984 Mathematics Road, Vancouver, BC V6T 1Z2, Canada \newline ORCiD: 0000-0002-3117-5028}


\date{\today}

\numberwithin{equation}{section}

\newcommand{\scrP}{\mathscr{P}}
\newcommand{\ord}{\text{ord}}

\newcommand{\vbar}{\overline{v}}

\newcommand{\Gr}{\text{Gr}}
\newcommand{\scrF}{\mathscr{F}}
\newcommand{\boldA}{\mathbf{A}}
\newcommand{\boldT}{\mathbf{T}}

\newcommand{\Tam}{\text{Tam}}
\newcommand{\frakP}{\mathfrak{P}}
\newcommand{\length}{\text{length}}
\newcommand{\coker}{\text{coker}}
\newcommand{\loc}{\text{loc}}
\newcommand{\tors}{\text{tors}}

\newcommand{\calL}{\mathcal{L}}

\DeclareFontFamily{U}{wncy}{}
\DeclareFontShape{U}{wncy}{m}{n}{<->wncyr10}{}
\DeclareSymbolFont{mcy}{U}{wncy}{m}{n}
\DeclareMathSymbol{\Sha}{\mathord}{mcy}{"58}

\makeatletter
\newcommand{\leqnomode}{\tagsleft@true}
\newcommand{\reqnomode}{\tagsleft@false}
\makeatother

\title[Cong. modular forms in anticyc. indefinite setting]{On $\lambda$-invariants of congruent modular forms in the anticyclotomic indefinite setting}
\begin{document}
\maketitle
\begin{abstract}
    We give a precise, computable formula for comparing $\lambda$-invariants between modular forms in the anticyclotomic indefinite setting where the Selmer groups have positive rank. This is an improvement of Hatley-Lei \cite{HL19, HL21} where the authors give a formula with incomputable error terms.
\end{abstract}
\section{Introduction} \label{sec:intro}
Let $K = \Q(\sqrt{-D})$ be an imaginary quadratic field of class number $h_K$ and $p$ be a prime number that is split in $K$ as $(p) = \frakp \overline{\frakp}$. Let $f \in S_{2r}(\Gamma_0(N))$ be a newform that is ordinary at $p$ (i.e. $a_p(f) \in \Z_p^\times$) whose coefficients lie in some finite extension $\mathfrak{F}$ of $\Q_p$. Assume that $f$ satisfies the Heegner hypothesis 
\[\label{Heeg}
\text{every prime } \ell \mid N \text{ is split in $K/\Q$} \\ 
\tag{Heeg.}\]
as well as the admissibility condition
\[\begin{cases} \label{adm}
	p  \text{ does not ramify in } \mathfrak{F} \\
	p \nmid 6 (2r - 1)! N \phi(N) h_K \\
	\text{if $r = 1$ then } a_p(f)^2 \not \equiv 1 \pmod{p}.
\end{cases} \tag{admiss.}\]

Attached to $f$ is a $2$-dimensional self-dual Galois representation $V_f$, which is the $r^{th}$ Tate-twist of the Galois representation constructed by Deligne. Denote this representation by
\[\rho_f: G_{\Q} \rightarrow GL_2(\mathfrak{F})\]
Let $\calO$ be the ring of integers in $\mathfrak{F}$ with uniformizer $\varpi$ and  
$\bar{\rho}_f: G_\Q \rightarrow GL_2(\kappa)$ be the residual representation of $\rho_f$, where $\kappa$ is the residue field of $\calO$. We also make the assumption
\[ \label{irr} \bar{\rho}_f \text{ is absolutely irreducible} \tag{irred.}\]

In this setting, the dual Selmer group over the anticyclotomic extension $X(K, \boldA_f)$ (defined in Section \ref{sec:prelim}) has rank $1$ as a module over the Iwasawa algebra \cite{How04b}. It is natural to ask how the Iwasawa $\mu$ and $\lambda$-invariants of the cotorsion part $X(K, \boldA_f)_{\tors}$ vary in a family of modular forms with isomorphic residual representations. 

When the dual Selmer group is torsion over the Iwasawa algebra, this type of congruence question has been studied extensively for both the cyclotomic \cite{GV00, EPW06} and anticyclotomic \cite{PW11, Kim17, CKL17} extensions. On the algebraic side of these papers, the prevalent method is to consider the \textit{imprimitive residual Selmer group} defined for $\bar{\rho}_f$ with local conditions at bad primes omitted. 

Alternatively, one may also look at \textit{the residual Selmer group} to obtain similar results in various contexts in both the cyclotomic \cite{LS18, NS23} and anticyclotomic \cite{HL21} settings. 

In fact, the paper \cite{HL21} studies congruence questions in the same context as our paper. However, the result in \cite{HL21} contains incomputable error terms for comparing the $\lambda$-invariants between residually isomorphic modular forms (see \cite[Theorem 4.6]{HL21}). The article addresses this gap with the following main result:

\begin{thm}[Theorem \ref{thm:lambda-inv}]
	Let $f_1 \in S_{2r_1}(\Gamma_0(N_1)), f_2 \in S_{2r_2}(\Gamma_0(N_2))$ be modular forms that satisfy \ref{Heeg}, \ref{adm} and \ref{irr}. Assume that $\bar{\rho}_{f_1} \simeq \bar{\rho}_{f_2}$ and $\mu(f_1)  = \mu(f_2) = 0$. Moreover, assume that both $X(K, \boldA_{f_1}), X(K, \boldA_{f_2})$ do not have any finite non-zero submodule. Then
	\begin{equation*}
		\lambda(f_1) + 2 \sum_{\ell \mid N_1 N_2} \lambda_\ell(f_1) = \lambda(f_2) + 2 \sum_{\ell \mid N_1 N_2} \lambda_\ell(f_2).
	\end{equation*}
    where $\lambda_\ell(f_i)$ are local constants defined in Definition \ref{defn:local-lambda}.
\end{thm}

For $i \in \{1, 2\}$, the invariants $\mu(f_i), \lambda(f_i)$ denote the Iwasawa $\mu$ and $\lambda$ invariants of $f_i$. The novelty in this article is the improvement of existing techniques in the literature that can only effectively address the cotorsion case. Not only do we obtain a precise formula for the $\lambda$-invariants but we also supplement our result with a concrete example in Section \ref{sec:ex} for two elliptic curves with isomorphic $p$-torsion as Galois modules. This computation relies on the author's previous result regarding non-existence of finite non-zero submodules of the dual Selmer group \cite[Theorem 4.19]{NS25}.

It is also worth mentioning that the authors of \cite{HL21} noticed an oversight in their result and provided an erratum attached to Appendix \ref{sec:appendix} of this article.

\section{Preliminaries} \label{sec:prelim}

For the $p$-adic representation $V = V_f$ attached to $f$, there exists a $G_\Q$-stable filtration
$$0 \rightarrow \scrF^+ V \rightarrow V \rightarrow \scrF^- V \rightarrow 0$$

where $\scrF^+V$ and $\scrF^-V$ are both $1$-dimensional representations. Let $T$ be a $G_\Q$-stable lattice in $V$ and let $A = V/T$. We also define $\scrF^+ T = T \cap \scrF^+ V$, $\scrF^- T = T/ \scrF^+ T$, and $\scrF^+ A = \scrF^+ V/ \scrF^+ T$, $\scrF^- A = A/ \scrF^+ A$.  

Recall the following local conditions above $p$, where $M$ is $A, A[\varpi],  V$ or $T$. Whenever it is important to specify the underlying modular form, we will instead use the notation $A_f, V_f$ and $T_f.$
 Let $F/K$ be a finite extension, and let $v$ be a prime of $F$. 
\begin{defn}
	The Greenberg local condition is defined as 
	\[H^1_{\Gr}(F_v, M) := \begin{cases}
		\ker \left(H^1(F_v, M) \rightarrow H^1(F_v^{nr}, \scrF^- M)\right) & \text{if } v \mid p,\\
		\ker \left(H^1(F_v, M) \rightarrow H^1(F_v^{nr}, M) \right) & \text{if otherwise}.
	\end{cases}\]
\end{defn}

\begin{defn}
	For $v \mid p$ and $\scrL_v \in \{\emptyset, \Gr, 0\}$, set
	$$H^1_{\scrL_v}(F_v, M):= \begin{cases}
		H^1(F_v, M) & \text{if } \scrL_v = \emptyset, \\
		H^1_{\Gr}(F_v, M) & \text{if } \scrL_v = \Gr, \\
		\{0\} & \text{if } \scrL_v = 0. 
	\end{cases}$$
\end{defn} 

Let $\Sigma$ be a finite set of primes of $K$ dividing the primes where $V$ is ramified as well as the primes dividing $p\infty$. We will denote by $F_\Sigma$ the maximal extension of $F$ unramified outside of the set of primes dividing the primes in $\Sigma$. 

\begin{defn} For a set of local conditions $\scrL  = \{\scrL_v\}_{v \mid p}$, we define
	\[\Sel_{\scrL} (F, M) = \ker \left(H^1(F_\Sigma/F, M) \rightarrow \prod_{v \mid \Sigma, v \nmid p} \frac{H^1(F_v, M)}{H^1_{\Gr}(F_v, M)} \times \prod_{v \mid p} \frac{H^1(F_v, M)}{H^1_{\scrL_v}(F_v, M)} \right)\]
	and let $X_{\scrL}(F, M)$ be the Pontryagin dual of $\Sel_{\scrL}(F, M)$.
\end{defn}

\begin{rmk}
    The Selmer group $\Sel_{\scrL}(F, A[\varpi])$ is called the residual Selmer group, for which $\Sigma$ can be restricted to the set of primes dividing $p\infty$ and the conductor of the residual representation $A[\varpi].$
\end{rmk}
Also let
\[ \Sel_{\scrL}(K_\infty, A) := \varinjlim_{K \subset F \subset K_\infty} \Sel_{\scrL}(F, A),\]
\[\Sel_{\scrL}(K_\infty, T) := \varprojlim_{K \subset F \subset K_\infty} \Sel_{\scrL}(F, T)\]
with compatible local conditions $\scrL$ where $F/K$ runs over finite subextensions contained in $K_\infty/K$.
 
 Denote by $\Lambda$ the Iwasawa algebra $\calO \llbracket \Gal(K_\infty/K) \rrbracket$ and define $\mathbf{T} := T \otimes \Lambda$, and $\mathbf{A} := A \otimes \Lambda^\ast$, where $\Lambda^\ast$ is the Pontryagin dual of $\Lambda$. If the modular form $f$ is not clear from context, we will use the  notation $\boldA_f$ and $\boldT_f$.
 
 There are isomorphisms \[\Sel_{\scrL}(K, \boldA) \simeq \Sel_{\scrL}(K_\infty, A)\] \[\Sel_{\scrL}(K, \boldT) \simeq \Sel_{\scrL}(K_\infty, T).\] 

Let the Pontryagin dual of $\Sel_{\scrL}(K, \boldA)$ be denoted by $X_\scrL(K, \boldA)$. The Iwasawa $\mu$ and $\lambda$-invariants of $X_{\scrL}(K, \boldA)$ as a $\Lambda$ will be denoted by $\mu_{\scrL}(f)$ and $\lambda_{\scrL}(f)$, respectively. 

\begin{rmk}
	The Selmer group $\Sel_{\Gr, \Gr}(K, \boldA)$ is also known as the Greenberg Selmer group \cite{Gre99}. In this case we will simply drop the local conditions and denote $\Sel_{\Gr, \Gr}(K, \boldA)$ as $\Sel(K, \boldA)$. We also apply this convention for other notation with local conditions $(\Gr, \Gr)$.
\end{rmk}

\begin{rmk}
	The Selmer group $\Sel_{\emptyset, 0}(K, \boldA)$ is known as the Bertolini-Darmon-Prasanna (BDP) Selmer group due to its connection with the Bertolini-Darmon-Prasanna $p$-adic $L$-functions (see \cite{Cas13}, \cite{BDP13}).
\end{rmk}


\section{Comparing $\lambda$-invariants of congruent modular forms}

For a modular form $f \in S_{2r}(\Gamma_0(N))$ satisfying the hypotheses \ref{Heeg}, \ref{adm} and \ref{irr}, we have $\mu(f) = \mu_{\emptyset, 0}(f) = 0$ \cite[Theorem 4.5]{HL21}. Moreover,
\begin{equation*}
	\ord_{\frakP}{X_{\emptyset, 0}}(K, \boldA) = 2 \cdot \length_{\frakP}(\coker(\loc_{\frakp})) + \ord_{\frakP}{X(K, \boldA)_{\tors}}
\end{equation*}where $\loc_{\frakp}$ is the localization map $\loc_\frakp: \Sel(K, \boldT) \rightarrow H^1_{\Gr}(K_\frakp, \boldT)$ and $\frakP$ is a height one prime ideal of $\Lambda$ \cite[Lemma A.4]{Cas17}.

Our result relies on the following

\begin{lemma} \label{lem:loc_p}
	Suppose that $X(K, \boldA)_{\tors}$ does not have any finite non-zero submodule. Then for a height one prime ideal $\frakP$ of $\Lambda$, $\length_{\frakP}(\coker(\loc_\frakp))$ depends only on the residual representation $\bar{\rho}_f.$
\end{lemma}
	\begin{proof}
	Denote by $\overline{\loc}_\frakp$ the residual map $ \Sel(K, \boldT/\varpi) \rightarrow H^1_{\Gr}(K_\frakp, \boldT/\varpi)$, which only depends on $\bar{\rho}_f$. We show that $\coker(\loc_\frakp)/\varpi \simeq  \coker(\overline{\loc}_\frakp)$.
	By global duality, there are  short exact sequences
	\begin{equation} \label{eqn:global-dual}
		{0} \rightarrow {\coker(\loc_\frakp)}  \rightarrow {X_{\emptyset, \Gr}(K, \boldA)}  \rightarrow {X(K, \boldA)} \rightarrow {0},
	\end{equation}
	\begin{equation*}
		{0} \rightarrow {\coker(\overline{\loc}_\frakp)}  \rightarrow {X_{\emptyset, \Gr}(K, \boldA[\varpi])}  \rightarrow {X(K, \boldA[\varpi])} \rightarrow {0}.
	\end{equation*}
	
	Our assumptions $\mu(f) = 0$ and $X(K, \boldA)$ does not have any finite non-zero submodule imply $\text{Tor}_1^{\Lambda}(X(K, \boldA), \Omega) = X(K, \boldA)[\varpi] = 0$. In particular, tensoring (\ref{eqn:global-dual}) with $\Omega = \Lambda/\varpi$ gives a short exact sequence
	\begin{equation*}
		{0} \rightarrow {\coker(\loc_\frakp)} \otimes \Omega  \rightarrow {X_{\emptyset, \Gr}(K, \boldA)} \otimes \Omega  \rightarrow {X(K, \boldA)} \otimes \Omega \rightarrow {0}
	\end{equation*}
    
	  We examine the following commutative diagram:
	\begin{center}
		\begin{tikzcd}[cramped, column sep = small]
			{0} \arrow[r] & {\coker(\loc_\frakp) \otimes \Omega} \arrow[r] \arrow[d, "\alpha"] & {X_{\emptyset, \Gr}(K, \boldA) \otimes \Omega} \arrow[r] \arrow[d, "\beta_{\emptyset, \Gr}"] & {X(K, \boldA) \otimes \Omega} \arrow[d, "\beta_{\Gr, \Gr}"] \arrow[r] & {0} \\
			{0} \arrow[r] & {\coker(\overline{\loc}_\frakp)} \arrow[r] & {X_{\emptyset, \Gr}(K, \boldA[\varpi])} \arrow[r]           & {X(K, \boldA[\varpi])} \arrow[r]           & {0}.
		\end{tikzcd}
	\end{center}
    
	Note that $\ker(\beta_{\emptyset, \Gr}), \ker(\beta_{\Gr, \Gr})$ are both finite and \begin{equation*}
		\dim_{\kappa} \ker(\beta_{\emptyset, \Gr}) = \dim_{\kappa} \ker(\beta_{\Gr, \Gr}) = \sum_{v \mid N} \dim_{\kappa} A_f^{I_v}/\varpi A_f^{I_v}
	\end{equation*} by Theorem $\ref{thm:control}$ in the Appendix.
	Hence, it suffices to show that $\alpha$ is surjective, and the conclusion follows. To this end, it is enough to show surjectivity of the map \[H^1_{\Gr}(K_\frakp, \boldT)/\varpi  \rightarrow H^1_{\Gr}(K_\frakp, \boldT/\varpi).\] 
    But this is  clear, as the cokernel of $H^1(K_\frakp, \scrF^+ \boldT)/\varpi  \rightarrow H^1(K_\frakp, \scrF^+ \boldT/\varpi)$ is $H^2(K_\frakp, \scrF^{+} \boldT) \simeq H^0(K_\frakp, \scrF^{-} \boldA)$ by local Tate duality, which is trivial under our assumptions.
\end{proof}

\begin{rmk}
    The sum \begin{equation*}
		\dim_{\kappa} \ker(\beta_{\emptyset, \Gr}) = \dim_{\kappa} \ker(\beta_{\Gr, \Gr}) = \sum_{v \mid N} \dim_{\kappa} A_f^{I_v}/\varpi A_f^{I_v}
	\end{equation*}
    in the proof of Lemma \ref{lem:loc_p} is a finite sum because every prime dividing $N$ is split in $K$ (the \ref{Heeg} hypothesis) and is therefore finitely decomposed in $K_\infty$.
\end{rmk}

\begin{defn} \label{defn:Euler}
	For a rational prime $\ell$, let $v$ be a prime $\ell$ in $K$ and define $P_v(f)(X) = \det(1 - X \cdot \gamma_v | V_f^{I_v})$ where $\gamma_v \in \Gal(K_\infty/K)$ is the Frobenius at $v$. Let $\scrP_v(f) \in \Z_p \llbracket \Gal(K_\infty/K) \rrbracket$ such that 
    $\scrP_v(f) = P_v(f)(N(v)^{-1} \gamma_v)$. If $\ell$ is split in $K$ then
    \[\scrP_v(f) = \begin{cases}
		1 - a_\ell(f) \ell^{-1} \cdot \gamma_v + \ell^{-1} \cdot \gamma_v^2  & \text{ if } \ell \nmid N, \\
		1 - a_\ell(f) \ell^{-1} \cdot \gamma_v & \text{ if } \ell \mid N.
	\end{cases}\]		
\end{defn}	

\begin{defn} \label{defn:local-lambda}
 When $\ell$ is split in $K$ as $(\ell) = v \bar{v}$, the $\lambda$-invariants of $\scrP_v(f)$ and $\scrP_{\bar{v}}(f)$ are equal and we will denote this $\lambda$-invariant by $\lambda_\ell(f).$   
\end{defn}

It is also important to note that in this case, 
\[\mu(\scrP_v(f)) = \mu(\scrP_{\vbar}(f)) = 0.\]
Moreover, $\gamma_v$ is a topological generator of $\Gal(K_{\infty, v}/K_v)$. There is an isomorphism $\Z_p \llbracket \Gal(K_{\infty, v}/K_v) \rrbracket \simeq \Z_p \llbracket T \rrbracket$ sending $\gamma_v \mapsto T.$ If we let $X: = \ell^{-1}(T + 1)$ then 
\[\scrP_v(f)(X) = \begin{cases}
1 - a_\ell(f) X + \ell \cdot X^2  & \text{ if } \ell \nmid N, \\
1 - a_\ell(f) X  & \text{ if } \ell \mid N.
\end{cases}\]
and the $\lambda$-invariant of $\scrP_v(f)$ as an element of $\Z_p \llbracket \Gal(K_{\infty, v}/K_v) \rrbracket$ is given by the multiplicity of $X = \ell^{-1}$ as a root of $\scrP_v(f)(X) \pmod{\varpi}$. The same statement applies to the conjugate prime $\vbar$. Finally, we note that $\scrP_v(f)(X)$ does not depend on $v$.

\begin{thm} \label{thm:lambda-inv}
	Let $f_1 \in S_{2r_1}(\Gamma_0(N_1)), f_2 \in S_{2r_2}(\Gamma_0(N_2))$ be modular forms that satisfy \ref{Heeg}, \ref{adm} and \ref{irr}. Assume that $\bar{\rho}_{f_1} \simeq \bar{\rho}_{f_2}$ and $\mu(f_1)  = \mu(f_2) = 0$. Moreover, assume that $X(K, \boldA_{f_1}), X(K, \boldA_{f_2})$ do not have any finite non-zero submodules. Then
	\begin{equation*}
		\lambda(f_1) + 2 \sum_{\ell \mid N_1 N_2} \lambda_\ell(f_1) = \lambda(f_2) + 2 \sum_{\ell \mid N_1 N_2} \lambda_\ell(f_2).
	\end{equation*}
\end{thm}
\begin{proof}
	Recall that 
	\begin{equation*}
		\ord_{\frakP}{X_{\emptyset, 0}}(K, \boldA_{f_i}) = 2 \length_{\frakP}(\coker(\loc_{\frakp})) + \ord_{\frakP}{X(K, \boldA_{f_i})_{\tors}}
	\end{equation*}
    for $i \in \{1, 2\}$ where $\coker(\loc_{\frakp})$ only depends on $\bar{\rho}_{f_1} \simeq \bar{\rho}_{f_2}$ by Lemma \ref{lem:loc_p}. The result now follows from the conclusion of \cite[Corollary 3.8]{LMX23}:
	\[\lambda_{\emptyset, 0}(f_1) + 2 \sum_{\ell \mid N_1 N_2} \lambda_\ell(f_1) = \lambda_{\emptyset, 0}(f_2) + 2 \sum_{\ell \mid N_1 N_2} \lambda_\ell(f_2).\]
\end{proof}

\begin{rmk}
Note that \cite[Corollary 3.8]{LMX23} in the proof also applies under the weak Heegner hypothesis where some primes dividing $N_1 N_2$ are inert in $K$.
\end{rmk}

\begin{rmk}
Compared to the corrected formula (\ref{eqn:appendix}), our formula does not involve the incomptable error terms $c_\calL(\cdot).$ Our formula also resembles previous results in the cyclotomic \cite{GV00} and the anticyclotomic definite setting \cite{PW11, CKL17}. 
\end{rmk}

For a modular form $f \in S_2(\Gamma_0(N))$ associated with an elliptic curve $E/\Q$ of conductor $N$ that satisfies \ref{Heeg}, let $y_K$ denote the Heegner point in $E(K)$. That is, $y_K$ is the image of $[(\C/\calO_K, \frakN^{-1}/\calO_K)] \in X_0(N)$ under the modular parametrization $\phi: X_0(N) \rightarrow E$ sending $i \infty$ to the origin of $E$ where $\frakN \subset \calO_K$ is an integral ideal such that $\calO_K = \frakN\overline{\frakN}$. When $y_K$ has infinite order, $E(K)$ necessarily has rank $1$. This is a classical result of Kolyvagin \cite{Kol90}, later geneneralized by Kolyvagin-Logachev \cite{KL89} and Howard \cite{How04b}. Under this hypothesis, we give explicit conditions under which the module $X(K, \boldA)$ does not admit non-zero finite submodules.

\begin{thm} \label{thm:finite-sub} Suppose $y_K$ has infinite order in $E(K)$.  Suppose \[[E(K) \otimes \Z_p : \Z_p(y_K \otimes 1)] = \prod_{\ell|N} c_\ell(E/\Q)^{(p)}\] and that one of the following holds:
	\begin{enumerate}
		\item The prime $p$ is split in $\calO_K$.
		\item The prime $p$ is inert in $\calO_K$ and $\widetilde{E}(k_v)[p] = 1$ for the unique prime $v | p$ in $K$.
	\end{enumerate}
Then the module $X(K, \boldA)$ does not have any non-zero finite submodule.
\end{thm}

As a result, one obtains the following corollary of Theorem \ref{thm:lambda-inv} for elliptic curves:

\begin{cor} \label{cor:ellcurves}
	Suppose that $E_1/\Q, E_2/\Q$ are elliptic curves of respective conductors $N_1, N_2$ which satisfy \ref{Heeg}, \ref{adm} and \ref{irr} and the hypotheses of Theorem \ref{thm:finite-sub}. Assume that $E_1[p] \simeq E_2[p]$ as $G_\Q$-modules and $\mu(E_1)  = \mu(E_2) = 0$. Then 
	\begin{equation*}
		\lambda(E_1) + 2 \sum_{\ell \mid N_1 N_2} \lambda_\ell (E_1) = \lambda(E_2) + 2 \sum_{\ell \mid N_1 N_2} \lambda_\ell (E_2).
	\end{equation*}
\end{cor}

\section{Numerical example} \label{sec:ex}

In the following example, we use SageMath \cite{sagemath} to verify the hypotheses of Corollary \ref{cor:ellcurves}. Our computation produces an instance where the $\lambda$-invariant is positive.

\begin{ex}
	Let $K = \Q(\sqrt{-51})$ and $p = 5$, which is split in $K$. Consider the elliptic curves $E_1$ and $E_2$ with Cremona labels $19a1$ and $817b1$, respectively. One can check that $E_1[p] \simeq E_2[p]$ by looking at the coefficients of the respective modular forms up to Sturm's bound \cite{Stu87}. The curve $E_1 = 19a1$ satisfies the hypotheses of \cite[Theorem 0.16]{MN19} and therefore $\Sel(E_1/K_\infty)$ is co-free, which implies $\lambda(E_1) = 0$. These hypotheses for $E_1$ can be verified using SageMath \cite{sagemath}:
	\begin{enumerate}[(a)]
		\item $E_1(K)[5] = 0$. In fact, $E_1(K)$ has trivial torsion,
		\item $p \nmid N_1 \cdot a_p(E_1) \cdot (a_p(E_1) - 1) \cdot c_{\Tam}(E_1/\Q)$,
		\begin{itemize}
			\item $a_p(E_1) = 2$,
			\item $\prod_{\ell \mid N_1} c_\ell (E_1/\Q) = 1$,
		\end{itemize}
		\item $y_K(E_1)$ has infinite order,
		\item $rk_{\Z} E_1(K) = 1$ and $\Sha(E_1/K)[5^\infty] = 0$.
	\end{enumerate}
	Both $E_1/K$ and $E_2/K$ satisfy the admissibility conditions in \cite{HL21}, namely they are ordinary at $p = 5$ and
	\begin{itemize}
		\item $5 \nmid 6N_i \phi(N_i) h_K$
		\item $a_5(E_i)^2 \not \equiv 1 \pmod{5}$,
	\end{itemize}
	for each $i \in \{1, 2\}$.
	Moreover, $E_2$ satisfies the hypotheses of \cite[Theorem 4.19]{NS25}:
	\[[E_2(K) : y_K(E_2)] = \prod_{\ell \mid N_2} c_\ell(E_2/\Q) = 10.\]
	Thus, $X(E_2/K_\infty)$ has no non-zero finite submodule. Applying Theorem \ref{thm:lambda-inv}, 
	\begin{equation} \label{eqn:cong}
		\lambda(E_1) + 2 \sum_{\ell \mid N_1 N_2} \lambda_\ell (E_1) = \lambda(E_2) + 2 \sum_{\ell \mid N_1 N_2} \lambda_\ell (E_2).
	\end{equation}
    
	Now, we compute the $\lambda$-invariants of the Euler factors $\lambda_\ell(E_i)$. In the same manner as \cite[Proposition 2.4]{GV00}, if $(\ell) = v \bar{v}$ then $\lambda_\ell(E_i) =  s_\ell d_\ell(E_i)$ where $s_\ell$ is the number of primes above $v$ in $K_\infty$  and $d_\ell(E_i)$ is the $\lambda$-invariant of $\scrP_v(f)$ as an element of $\Z_p \llbracket \Gal(K_{\infty, v}/K_v) \rrbracket$ (see Definition \ref{defn:Euler} and the discussion before Theorem \ref{thm:lambda-inv}).
    
	For $\ell = 19$, both $E_1$, $E_2$ have split-multiplicative reduction and $\ell^{-1}$ is not a root of $1 - X \pmod{5}$ so $\lambda_\ell(E_i) = 0$ for $i \in \{1, 2\}$. 
    
    For $\ell = 43$, $E_1$ has good ordinary reduction while $E_2$ has split-multiplicative reduction. It is again straightforward that $ \lambda_\ell(E_2) = 0.$ Now, $1 - a_{43}(E_1) X + 43 X^2 = 1 + X + 43X^2$ admits $\ell^{-1} \pmod{p}$ as a root with multiplicity $d_\ell(E_1) = 1$; $s_\ell$ can be computed by understanding prime decomposition in the anticyclotomic extension using \cite[Theorem 2(a)]{Bri07}. We show that $s_\ell = 1$ for our example. The imaginary quadratic field $K$ has class number $h_K = 2$ and we write $\ell^{h_K}$ in the form $a^2 + ab + ((D + 1)/4) b^2$ where $a = 12$ and $b = 11.$ Let $\omega = (\sqrt{-D} + 1)/2$. For our class number $h_K = 2$, $s_\ell$ is simply given by the largest power of $p$ dividing $b^\ast/ p$ where $(a + b \omega)^{p - 1} = a^\ast + b^\ast \omega$. In this case, $b^\ast = -952105$ and therefore $s_\ell = 1$. Hence, $\lambda_\ell(E_1) = 1.$

    We may then conclude from identity (\ref{eqn:cong}) that $\lambda(E_2) = 2$.
\end{ex}

\appendix \section{(by Jeffrey Hatley and Antonio Lei)} \label{sec:appendix}

There is a mistake in \cite[Section 4]{HL21}; note that no other sections of the paper are affected by this mistake. For a discussion of a similar oversight, see \href{https://www.math.ntu.edu.tw/~mlhsieh/research/erratum.pdf}{this note}.

The mistake occurs in Theorem 4.4 of \cite{HL21}; namely, the conclusion that the map
\begin{equation}\label{eq:isom}
	\Sel_\mathcal{L}(K_n,A_m) \rightarrow \Sel_\mathcal{L}(K_n,A)[\varpi^m]
\end{equation}
is an isomorphism might not hold when there are primes $\ell \mid N$ for which $A[\varpi]$ is unramified; in general, this map is only an injection with finite cokernel. The existence of primes dividing $N$ for which the residual Galois representation is unramified may occur when comparing two modular forms $f$ and $g$ of different levels with isomorphic residual representations, which is precisely the context of \cite[$\S$~4]{HL21}. Because of this, the $\lambda$-invariant formula in Theorem 4.6 is incorrect as stated.

By a result of \cite{Car86}, the representation attached to $f$ satisfies
\[
\rho_f|_{G_{\Q_\ell}}\sim\begin{pmatrix}
	\epsilon_\cyc&*\\0&1
\end{pmatrix},
\]
where $\ell$ is a prime not dividing $pN$. In the proof of \cite[Theorem 4.4]{HL21}, we erroneously claimed that this implies that $A^{I_v}$ is divisible, where $v$ is a prime of $K$ lying above $\ell$. For example, if $*$ represents a function on $I_v$ whose image lies inside $\varpi \calO$, but not $\varpi^2\calO$, then $$ A^{I_v}=\mathfrak{F}/\calO+\frac{1}{\varpi}\calO/\calO,$$
which is not divisible.

We record a corrected version of Theorem 4.4:

\begin{thm} \label{thm:control}
    The map \[\Sel_{\scrL}(K, \boldA[\varpi]) \rightarrow \Sel_{\scrL}(K, \boldA)[\varpi]\] is injective and has finite cokernel of $\kappa$-dimension
    \[\sum_{v\in\Sigma_0}\dim_\kappa A^{I_{\infty, v}}/\varpi A^{I_{\infty, v}}\]
\end{thm}
\begin{proof}
Examine the diagram
\begin{center}
		\begin{tikzcd}[cramped, column sep = small]
            {H^1(K_\infty, A[\varpi])} \arrow[r] \arrow[d] & \prod_{v \mid \Sigma, v \nmid p} H^1(K_{
            \infty, v}^{nr}, A[\varpi]) \times \prod_{v \mid p} \frac{H^1(K_{\infty, v}, A[\varpi])}{H^1_{\scrL_v}(K_{\infty, v}, A[\varpi])} \arrow[d, "\gamma"]
            \\
            {H^1(K_\infty, A)[\varpi]} \arrow[r]           & \prod_{v \mid \Sigma, v \nmid p} H^1(K_{
            \infty, v}^{nr}, A) \times \prod_{v \mid p} \frac{H^1(K_{\infty, v}, A)}{H^1_{\scrL_v}(K_{\infty, v}, A)}. 
		\end{tikzcd}
	\end{center}
    
    The tautological exact sequence $0 \rightarrow A[\varpi] \rightarrow A \xrightarrow{\times \varpi} A \rightarrow 0$ implies that the kernel of $H^1(K_\infty, A[\varpi]) \rightarrow H^1(K_\infty, A)[\varpi]$ is isomorphic to the image of $H^0(K_\infty, A)$ in the corresponding long exact sequence in cohomology. The cohomology $H^0(K_\infty, A)$ vanishes under hypothesis \ref{irr}. Hence, the map $H^1(K_\infty, A[\varpi]) \rightarrow H^1(K_\infty, A)[\varpi]$ is injective. It follows that $\Sel_{\scrL}(K, \boldA[\varpi]) \rightarrow \Sel_{\scrL}(K, \boldA)[\varpi]$ is also injective. 
    
    Moreover, it is immediate from the long exact sequence in cohomology that $H^1(K_\infty, A[\varpi]) \rightarrow H^1(K_\infty, A)[\varpi]$ is surjective. The cokernel of 
    \[\Sel_{\scrL}(K, \boldA[\varpi]) \rightarrow \Sel_{\scrL}(K, \boldA)[\varpi]\]
    is isomorphic to the kernel of $\gamma$ via the snake lemma. At each prime $v \in \Sigma,$, denote by $\gamma_v$ the component of $\gamma$ at the prime $v$. The kernel of each $\gamma_v$ is given by
    \[\ker(\gamma_v) = \begin{cases} A^{I_{\infty, v}} /\varpi A^{I_{\infty, v}} & \text{if } v \nmid p, \\
    H^0(K_{\infty, v}, A) & \text{if } v \mid p, \text{ or} \\
    H^0(K_{\infty, v}, \scrF^{-} A) & \text{if } v \mid p.
    \end{cases}\]

    It is known that $H^0(K_v, \scrF^{-} A) = 0$ \cite[Lemma 1.8]{CH15}, which uses the hypothesis $a_p(f)^2 \not \equiv 1 \pmod{p}$ for weight $2$ forms. For the remaning local condition above $p$, we use the description
    \[\rho_{f \mid {G_{\Q_p}}} \sim \begin{pmatrix} \chi_p^{-1}  \epsilon^{r} & \ast \\ 0 & \chi_p \epsilon^{1 - r}\end{pmatrix}\]
    where $\chi_p$ is unramified and $\epsilon$ is the $p$-cyclotomic character \cite{CH15} to conclude that $H^0(K_{\infty, v}, A) = 0.$
    
\end{proof}
    
A corrected version of Theorem 4.4 should read as follows:
\[
\dim_\kappa\Sel_\mathcal{L}(K_\infty,A_\star)[\varpi]+\sum_{v\in\Sigma_0}\dim_\kappa A_\star^{I_v}/\varpi A_\star^{I_v}=\dim_\kappa\Sel(K_\infty,A_\star[\varpi])
\]
for $\star\in\{f,g\}$.

On the one hand, $\overline\rho_f\cong\overline\rho_g$ implies that $\dim_\kappa\Sel(K_\infty,A_f[\varpi])=\dim_\kappa\Sel(K_\infty,A_g[\varpi])$.
On the other hand, from Proposition 3.2 it follows that 
\[
\dim_\kappa\Sel_\mathcal{L}(K_\infty,A_\star)[\varpi]=\lambda(\Sel_\mathcal{L}(K_\infty,A_\star))+c_\calL(\star).
\]
Therefore, combining these equations gives:

\begin{equation} \label{eqn:appendix}
\lambda_\calL(f)+c_\calL(f)+\sum_{v\in\Sigma_0}\dim_\kappa A_f^{I_v}/\varpi A_f^{I_v}=
\lambda_\calL(g)+c_\calL(g)+\sum_{v\in\Sigma_0}\dim_\kappa A_g^{I_v}/\varpi A_g^{I_v}.
\end{equation}
\section*{Acknowledgements}
The author would like to thank his supervisors Antonio Lei and Sujatha Ramdorai for suggesting the problem in this paper. Sujatha raised the question of how it can further be investigated from the lens of Hida families, which is an interesting topic we will reserve for the future.

\bibliographystyle{plain}
\bibliography{refs.bib}          
%
%

\end{document}